\documentclass[11pt]{amsart}
\usepackage{graphicx, epstopdf}
\usepackage{amssymb, verbatim, color}
\newtheorem{theorem}{Theorem}
\newtheorem{lemma}[theorem]{Lemma}
\newtheorem{prop}[theorem]{Proposition}

\newcommand{\Id}{\mathrm{Id}}
\newcommand{\fracpim}{\frac{\pi i}{m}}
\newcommand{\fractt}{\frac{2\pi i}{3}}

\newcommand{\qb}{{\mathbb C}^2}

\begin{document}

\title[]{Qubit representations of the braid groups from generalized Yang-Baxter matrices}

\author{Jennifer F. Vasquez}
\email{jennifer.vasquez@scranton.edu}
\address{The Mathematics Department\\
The University of Scranton\\
Scranton, PA 18510}

\author{Zhenghan Wang}
\email{zhenghwa@microsoft.com}
\address{Microsoft Station Q and Dept of Mathematics\\
    University of California\\
    Santa Barbara, CA 93106}

\author{Helen M. Wong}
\email{hwong@carleton.edu}
\address{Carleton College\\
Department of Mathematics and Statistics\\
Northfield, MN 55057
}

\thanks{The second author is partially supported by NSF grant DMS-1411212, and the third author by NSF grants DMS-1105692 and DMS-1510453.  The authors thanks Matt Hastings for valuable communications.}

\begin{abstract}
Generalized Yang-Baxter matrices sometimes give rise to braid group representations.  We identify the exact images of some qubit representations of the braid groups from generalized Yang-Baxter matrices obtained from anyons in the metaplectic modular categories.
\end{abstract}
\maketitle

\section{Introduction}

A generalized Yang-Baxter (gYB) matrix is an invertible $8\times 8$ matrix $R: {(\qb)}^{\otimes 3}\rightarrow {(\qb)}^{\otimes 3}$ such that
$$ (R\otimes I)(I\otimes R)(R\otimes I)=(I\otimes R)(R\otimes I)(I\otimes R),$$
where $I$ is the identity operator on $\qb$.  As in quantum information, we will refer to $\qb$ as a {\it qubit}.  This generalization of the Yang-Baxter equation, inspired by quantum information, is proposed in \cite{RZWG}, and referred to as the $(2,3,1)$-generalized in \cite{KW}.  One application of a gYB matrix is to give rise to new representations of the braid groups $\mathcal B_n$ on $(n+1)$-qubits ${(\qb)}^{\otimes (n+1)}$ by sending the standard braid generator $\sigma_i$ to
$$ R_{\sigma_i}=I^{\otimes (i-1)}\otimes R \otimes I^{\otimes (n-i-1)}.$$ 
But $R_{\sigma_i}$'s do not necessarily satisfy the far commutativity relation automatically.  Therefore, we need to check the commutativity in order to have braid group representations from gYB matrices.  We will refer to a braid group representation from a gYB matrix a {\it qubit} braid group representation.

One systematic way to find gYB matrices is to use weakly-integral anyons \cite{KW}.  An interesting class of weakly-integral anyons are those from the metaplectic modular categories related to parfermion zero modes \cite{HNW2}.  In \cite{HNW1}, the authors considered the braid group representations from the anyon types $Y_i$ in the metaplectic modular categories $SO(m)_2, m\geq 3$ odd.  But the authors did not exactly identify the images of the resulting qubit representations of the braid groups.  In this note, we completely identify the images for the case of odd $m$.

The explicit representation matrices can be used as quantum gates to set up quantum computation models.  One particular way would be to allow some qubits in the $\mathcal B_n$ representation spaces to be ancillas.  Since the braid representations have finite images, therefore the braiding gates alone cannot be universal for quantum computation.  It would be interesting to see if we can obtain universality by supplementing braiding gates with measurements as in \cite{CHW,CW}.

\section{Qubit braid group representations and their images}

Let $\mathcal B_n$ be the braid group on $n$ strings, generated by the elementary braids $\sigma_1, \sigma_2, \ldots, \sigma_{n-1}$.    We consider a representation $\rho_{R}: \mathcal B_n \to \mathrm{End} ( (\mathbb C^2)^{\otimes (n+1)}) $ considered in \cite{HNW1}.    We define $\rho_R$  and express it using the standard operators.

\subsection{Definition of the gYB representation $\rho_R$}
  Let $m\geq 3$ be an odd integer. Let $\nu = -1$ if $m = 3$,  and   $\nu = + 1$ if  $m \geq 5$.   Then  $R$  (which was denoted by $R_{Y_1}$ in \cite{HNW1}) is the $8 \times 8$ gYB matrix
\[
\begin{pmatrix}
 \nu \cos( \frac{\pi}{m}) 	& 0 					& i \sin ( \frac{\pi}{m}) 	& 0 \\
0					&   - i \sin( \frac{\pi}{m}) 	& 0 				& \cos ( \frac{\pi}{m}) \\
 i \sin( \frac{\pi}{m}) 		& 0 					& \nu \cos ( \frac \pi m) & 0\\
0					&  \cos( \frac{\pi}{m}) 		& 0 				& -i \sin ( \frac{\pi}{m}) 	 \\
\end{pmatrix}
\oplus
\begin{pmatrix}
- i \sin( \frac{\pi}{m}) 		& 0 					&  \cos ( \frac{\pi}{m}) & 0\\
0					& \nu \cos( \frac{\pi}{m}) 	& 0 				& i \sin ( \frac{\pi}{m}) \\
	\cos( \frac{\pi}{m}) 	& 0 					& -i \sin ( \frac \pi m) 	& 0 \\
0					&   i \sin( \frac{\pi}{m}) 	& 0 				& \nu \cos ( \frac{\pi}{m}) \\
\end{pmatrix}
,
\]
\smallskip
where the $\oplus$ is the block sum of matrices.  Here, we use the lexicographical convention for the order of the eight 3-qubit basis elements.  

Let $n \geq 2$.  The qubit representation $\rho_{R}$ is the representation  of $\mathcal B_n$ on $(n+1)$-qubits such that
\[ \rho_{R}(\sigma_{i}) = I^{\otimes(i-1)} \otimes R \otimes I^{\otimes(n-i-1)} \]
for every $i = 1, \ldots, n-1$ (earlier referred to as $R_{\sigma_{i}}$).
Since $\mathcal B_n$ is generated by the elementary braids $\sigma_1 , \ldots, \sigma_{n-1}$,  this determines the action of $\rho_{R}$ for all elements of $\mathcal B_n$.  The far commutativity can be checked directly, therefore, we have a qubit representation of the braid group.

The matrices $U_{i-1, i, i+1}$ in \cite{HNW1} correspond to our $\rho_R(\sigma_{i-1})$; we follow their convention for the sake of symmetry.  For the remainder of the paper, we take $i = 2, \ldots, n$.    In particular, $\rho_R(\sigma_{i-1})$ acts  on the $(i-1, i, i+1)$-qubits using $R$ and leaves all the others the same.  
 
\subsubsection{Standard gates} 
Let $X_i$ be the Pauli gate that changes the $i$-th qubit.   Let $Z_i$ be the Pauli gate that negates the qubit  if the $i$-th qubit is nonzero.  For example,
\[ X_2(|abc\rangle ) = | a \bar b c\rangle  \; \mbox{  and  } \; Z_1Z_3(|abc\rangle ) = \begin{cases} |abc\rangle & \mbox{ if } a = c\\  -|a b c\rangle & \mbox{ if } a \neq c \end{cases}. \]

Let $\Lambda_{XOR}^2 NOT$ be the $XOR$ controlled 3-qubit gate defined on the 3-qubit $|abc\rangle$:
\[ \Lambda_{XOR}^2 NOT(|abc\rangle)=
\begin{cases} |abc\rangle & \mbox{ if } a = c \\
|a \bar bc\rangle & \mbox{ if } a \neq c \end{cases}. \]

Let $NOT_i$ (or  $NOT_{i-1, i, i+1}$) be the operator $I^{\otimes(i-2)} \otimes \Lambda_{XOR}^2 NOT \otimes I^{\otimes(n-i-2)}$.  In particular, $NOT_i$ is defined for $2 \leq i \leq n$.  It acts like  $\Lambda_{XOR}^2 NOT$  on the consecutive $(i-1, i, i+1)$-qubits and leaves all the others unchanged.  Whereas $Z_{i-1}Z_{i+1}$ negates the qubit iff the $(i-1)$th and $(i+1)$th qubits disagree,  $NOT_i$ reverses the $i$th qubit iff the $(i-1)$th and $(i+1)$th qubits disagree.
 
\bigskip 

Note the following well-known commutativity properties between the Pauli gates and the $NOT_i$ operators.  
\begin{lemma}  \label{COMMidentities}  
\begin{enumerate}
\item $[X_i, X_j] = 0$ $\forall ~i, j$. 
\item $[Z_i, Z_j] = 0$ $\forall ~i, j$.
\item $X_i Z_i = - Z_i X_i$ and  $[X_i,Z_j]=0$ $\forall~ i \neq j$. 
\item $Z_i NOT_i = (Z_{i-1} Z_{i+1}) NOT_i  Z_i $ and $[Z_i,NOT_j]=0$ $\forall ~ i \neq j$.
\item $ NOT_{i}  X_{i-1}  = X_{i-1} X_{i} NOT_{i}$, $NOT_{i} X_{i+1}   =   X_i X_{i+1} NOT_{i} $, and $[NOT_i,X_j]=0$ $\forall ~ j \neq i-1, i+1$.
\end{enumerate}
\end{lemma}
\bigskip

The $NOT_i$ operators also satisfy the following relations: 

\begin{lemma} \label{NOTidentities}
\begin{enumerate}
\item $ NOT_i^2 = \Id $.
\item $ NOT_i NOT_{i+1} NOT_i = NOT_{i+1} NOT_i NOT_{i+1}. $
\end{enumerate}
\end{lemma}

A variation of the next proposition features prominently in the characterization of the image of $\rho_R$.   We present it separately, as it may be of independent interest. 

\begin{prop} \label{NOTgroup}
The group $G$ generated by $NOT_2, \ldots, NOT_n$  is isomorphic to the symmetric group $S_n$.  
\end{prop}

\begin{proof}
Let $(i-1, i)$ denote the element in $S_n$ that transposes the $(i-1)$th and $i$th places.   As $S_n$ is generated by such transpositions, we may define a map  $\phi: S_n \to G$ by 
$\phi( (i-1, i)) =NOT_i$ for $2 \leq i \leq n$.

Lemma \ref{NOTidentities} immediately implies that $\phi$ is a surjective homomorphism. To show that $\phi$ is injective, first note that $\ker \phi$ is a normal subgroup of $S_n$. For $n \geq 5$, $S_n$ is solvable.  So $ker(\phi) \in \{\{e\},S_n, A_n\}$.
Since the image of $\phi$ is $G$ and obviously $|G| >2$ for this choice of $n$, $\ker(\phi) = \{ e \}$.  Therefore, $\phi$ is an isomorphism for $n \geq5$.
Of the remaining cases, the one for $n=2$ is obvious, since both $S_n$ and $G$ are isomorphic to $\mathbb Z_2$.  For $n=3$ and $n=4$, we proceed similarly to the argument above for $n \geq 5$, except that we need to find explicit, distinct elements of $G$ to show $|G|>2$ for $n=3$ and $|G|>6$ for $n=4$.   For $n=3$, we check that $NOT_2$, $NOT_3$, and $NOT_2 NOT_3$ are distinct by comparing their actions on the 4-qubit $|0100\rangle$.  For $n=4$,  we need at least seven distinct elements.  We check that $NOT_2$, $NOT_3$, $NOT_4$, $NOT_2 NOT_3$, $NOT_3 NOT_4$, $NOT_3 NOT_2$, $NOT_4 NOT_3$ act distinctly on the 4-qubit $|0110 \rangle$.
\end{proof} \bigskip
%%%%%%%%%%%%%%%%%%%%

\subsubsection{Writing the gYB representation in terms of standard gates}
We express the action of $R$ on $3$-qubits as
\[ R(|abc\rangle ) =  \begin{cases}   \nu \cos(\frac{\pi}{m}) |abc\rangle + i \sin (\frac{\pi}{m}) |a \bar b c \rangle & \mbox{ if } a = c \\  - i \sin (\frac{\pi}{m}) |abc\rangle + \cos (\frac{\pi}{m}) |a \bar b c\rangle & \mbox{ if } a \neq c \end{cases}.\]
Direct computation then shows
\[  R =
  \begin{cases}
   e^{\fractt \; X_2} \cdot Z_1Z_3  \Lambda_{XOR}^2 NOT , & \text{for } m=3 \\
     e^{   \frac{\pi i}{m} Z_1X_2Z_3} \cdot \Lambda_{XOR}^2 NOT, & \text{for } m \geq 5
  \end{cases}.\]

Hence for $2 \leq i \leq n$
$$ \rho_R(\sigma_{i-1}) =
  \begin{cases}
 e^{\frac{2 \pi i }{3} \; X_i} \cdot  Z_{i-1}Z_{i+1}  NOT_i, & \text{for } m=3 \\
      e^{   \frac{\pi i}{m} Z_{i-1}X_i Z_{i+1}} \cdot NOT_i & \text{for } m \geq 5
  \end{cases}.$$

Note that there was an error in \cite{HNW1} for the $m=3$ case.
\\
\bigskip

%%%%%%%%%%%%%%%%%%%%%%%%%%%%%%%%%

\subsection{The image of the qubit representation when $m \geq 5$ is odd}

\begin{theorem} \label{TheThm}
For $m\geq  3$ odd, the image of $\rho_R$ is isomorphic to  $\mathbb Z_m^{\frac{n(n-1)}{2}} \rtimes S_n.$
\end{theorem}
\bigskip

We prove this theorem in a series of lemmas.  Assume $m$ is odd from now on.   Following \cite{HNW1}, for $ 2 \leq i \leq n$, define
$$ H_i =
  \begin{cases}
   X_i, & \text{for } m=3 \\
     Z_{i-1}Z_{i+1} X_i, & \text{for } m \geq 5  
  \end{cases}.$$\\
For $k \leq l$, define the product of consecutive $H$'s as
\[ S_{k,l} = H_k H_{k+1} \cdots H_{l}. \]
\bigskip

\begin{lemma}
The image of $\rho_R$ is generated by:
\begin{itemize}
\item  (when $m=3$)  $ \{  e^{\fractt(-1)^{l-k} S_{k,l}} \; | \; 2 \leq k \leq l \leq n \} $ and $ \{ Z_{k-1} Z_{k+1}  NOT_k \; | \; 2 \leq k \leq n \}. $
\item (when $m  \geq 5$ odd) $  \{ - e^{\fracpim (-1)^{l-k} S_{k,l}} \; | \; 2 \leq k \leq l \leq n \} $  and $ \{ -NOT_k \; | \; 2 \leq k \leq n \}. $
\end{itemize}
\end{lemma}

\begin{proof}

{\bf Case for $m = 3$:} Recall that $\rho_R(\sigma_{k-1}) =   e^{\fractt X_{k}} Z_{k-1}Z_{k+1}NOT_k $.   It follows from Lemma \ref{COMMidentities} that $Image(\rho_R)$ also contains
\[ ( e^{\fractt X_{k}} Z_{k-1}Z_{k+1}NOT_k  )^3  =Z_{k-1}Z_{k+1}NOT_k  \]
and
\[ ( e^{\fractt X_{k}} Z_{k-1}Z_{k+1}NOT_k  )^4 =  e^{\fractt X_{k}}. \]
Recall $S_{k,l} = X_k X_{k+1} \cdots X_l$.    Induction shows the following is also in $Image(\rho_R)$:
\[ ( Z_l Z_{l+2} NOT_{l+1})( e^{\fractt  (-1)^{l-k} S_{k,l}}) (Z_l Z_{l+2} NOT_{l+1}) =   e^{ \fractt  (-1)^{l+1-k} S_{k,l+1}}. \]
Again Lemma \ref{COMMidentities} is used to rearrange the operators.  Thus all the elements $e^{\fractt (-1)^{l-k} S_{k,l}}$ and $Z_{k-1} Z_{k+1}  NOT_k $ are contained in $Image (\rho_R)$.
Containment in the other way is obvious, because the image of each braid element $\rho(\sigma_{k-1}) =  e^{\frac{2 \pi i}{m} X_k} \cdot Z_{k-1}Z_{k+1}NOT_k$ can be written as a product of  $ e^{\frac{2 \pi i}{m} X_k}$ and $Z_{k-1}Z_{k+1}NOT_k$.
\\

{\bf Case for $m \geq 5$ odd:}  Here  $\rho_R(\sigma_{k-1}) =   e^{\fracpim H_{k}} NOT_k $, where $H_k = Z_{k-1} X_k Z_{k+1}$.  Thus
\[ (e^{\fracpim H_{k}} NOT_k  )^m  = - NOT_k \]
and
\[ (  e^{\fracpim H_{k}} NOT_k)^{m+1} = - e^{\fracpim H_{k}} \]
are also in the image of $\rho_R$.
Moreover, with $S_{k,l} = H_k H_{k+1} \cdots H_l$,
\[ (- NOT_{l+1})( - e^{\fracpim (-1)^{l-k} S_{k,l}}) (- NOT_{l+1}) =  - e^{\fracpim (-1)^{l+1-k} S_{k,l+1}}. \]
Arguing similarly to the $m=3$ case, we see that $Image(\rho_R)$ is generated by all the $ e^{\fracpim (-1)^{l-k} S_{k,l}}$ and $- NOT_k$.
\end{proof} \bigskip

%%%%%%%%%%%%%%%%%%%%

We distinguish between the two kinds of generators.  Define groups $\Gamma_{skl}$ and $\Gamma_{not}$ as follows:
\begin{itemize}
\item (when $m = 3$) \\
$\Gamma_{skl} $ to be the group generated by $\{  e^{\fractt (-1)^{l-k} S_{k,l}} \; | \; 2 \leq k \leq l \leq n  \} $ and \\
  $\Gamma_{not} $ to be the group generated by  $\{ Z_{k-1} Z_{k+1}  NOT_k  \; | \; 2 \leq k \leq n \}.$ \\

\item (when $m \geq 5$ odd) \\
 $\Gamma_{skl} $ to be the group generated by $\{ - e^{\fracpim (-1)^{l-k} S_{k,l}} \; | \; 2 \leq k \leq l \leq n \} $ and \\
 $\Gamma_{not} $ to be the group generated by  $\{ -NOT_k \; | \; 2 \leq k \leq n \}.$  \\
\end{itemize}

\bigskip

\begin{lemma}
The image of $\rho_R$ is a semi-direct product $\Gamma_{skl} \rtimes \Gamma_{not}$.
\end{lemma}

\begin{proof}
Note that the intersection is  $\Gamma_{skl} \cap \Gamma_{not} = \{ e \}$.
To show that we have a semi-direct product, we need to prove two things.  Firstly, that every element of $Image(\rho_R)$ is a product of an element of $\Gamma_{skl}$ with an element of $\Gamma_{not}$.   And secondly, that conjugation by elements of $\Gamma_{not}$ is an automorphism of $\Gamma_{skl}$.   Both of these can be shown from the following identities:

When $m = 3$, where $S_{k, l} = X_k X_{k+1} \cdots X_l $,
\[ (Z_{j-1} Z_{j+1} NOT_j) S_{k,l} (Z_{j-1} Z_{j+1}NOT_j)
= \begin{cases}
-S_{k-1, l} 		& \mbox{ when }  j=k-1 \\
-S_{k+1, l} 	& \mbox{ when }  j = k \mbox{ and } k < l \\
-S_{k, l-1}		&  \mbox{ when } j = l \mbox{ and } k < l \\
-S_{k, l+1}		& \mbox{ when } j = l + 1\\
S_{k, l}		& \mbox{ otherwise}
\end{cases}
\]
So
\[ (Z_{j-1} Z_{j+1} NOT_j) e^{\fractt (-1)^{l-k} S_{k,l}} (Z_{j-1} Z_{j+1}NOT_j)
= \begin{cases}
e^{\fractt (-1)^{l-(k-1)} S_{k-1,l}} 		& \mbox{ when }  j=k-1\\
e^{\fractt(-1)^{l-(k+1)} S_{k+1,l}}  	&\mbox{ when }    j = k \mbox{ and } k < l \\
e^{\fractt(-1)^{(l-1) -k} S_{k,l-1}} 		& \mbox{ when }  j = l \mbox{ and } k < l \\
e^{\fractt(-1)^{(l+1) -k} S_{k,l+1}} 	& \mbox{ when }  j = l + 1\\
e^{\fractt(-1)^{l-k} S_{k,l}} 	& \mbox{ otherwise}
\end{cases}
\]
In particular, conjugating a generator of $\Gamma_{skl}$ by a generator of $\Gamma_{not}$ is again a generator of $\Gamma_{skl}$.  It immediately follows that conjugation by $\Gamma_{not}$ is an automorphism of $\Gamma_{skl}$.   And with a bit more work, the same identities show that every element of $Image(\rho_R)$ is a product of an element of $\Gamma_{skl}$ with an element of $\Gamma_{not}$. Since $\Gamma_{skl}$ is a normal subgroup, $Image(\rho_R) = \Gamma_{skl} \rtimes \Gamma_{not}.$

When $m \geq5$ odd,  where $H_k = Z_{k-1} X_k Z_{k+1}$ and  $S_{k,l} = H_k H_{k+1} \cdots H_l$, the same identities are true.  Namely,
\[ (-NOT_j) S_{k,l} (-NOT_j)
= \begin{cases}
-S_{k-1, l} 		&  j=k-1\\
-S_{k+1, l} 	&  j = k \mbox{ and } k < l \\
-S_{k, l-1}		&  j = l \mbox{ and } k < l \\
-S_{k, l+1}		&  j = l + 1\\
S_{k, l}		& \mbox{otherwise}
\end{cases}
\]
\[ (- NOT_j)(- e^{i  \fracpim (-1)^{l-k} S_{k,l}} )(-NOT_j)
= \begin{cases}
-e^{\fracpim (-1)^{l-(k-1)} S_{k-1,l}} 		& \mbox{ when }  j=k-1\\
-e^{\fracpim (-1)^{l-(k+1)} S_{k+1,l}}  	&\mbox{ when }    j = k \mbox{ and } k < l \\
-e^{\fracpim (-1)^{(l-1) -k} S_{k,l-1}} 		& \mbox{ when }  j = l \mbox{ and } k < l \\
-e^{\fracpim (-1)^{(l+1) -k} S_{k,l+1}} 	& \mbox{ when }  j = l + 1\\
-e^{\fracpim (-1)^{l-k} S_{k,l}} 	& \mbox{ otherwise}
\end{cases}
\]
So for the same reasons as in the $m=3$ case, $Image(\rho_R) = \Gamma_{skl} \rtimes \Gamma_{not}$ when $m\geq 5$ odd.
\end{proof} \bigskip

%%%%%%%%%%%%%%%%%%%%

\begin{lemma}
$\Gamma_{not}$ is isomorphic to the symmetric group $S_n$. 
\end{lemma}

\begin{proof}
The proof is essentially the same as in Lemma \ref{NOTgroup} with a few minor tweaks.   Specifically, define $\phi:S_n \to \Gamma_{not}$ so that
 \begin{itemize}
 \item (when $m=3$) $\phi( (k-1, k)) = Z_{k-1}Z_{k+1}NOT_k$
 \item (when $m \geq 5$) $\phi( (k-1, k)) =- NOT_k$
 \end{itemize}
 Lemmas \ref{COMMidentities} and \ref{NOTidentities} imply that $\phi$ is a surjective homomorphism.  The proof of injectivity is identical for $n=2$ and $n \geq 5$.  For the case $n=3$:  use   $Z_1 Z_3NOT_2$, $Z_2 Z_4NOT_3$, and $(Z_1Z_3NOT_2)(Z_2Z_4NOT_3)$ for $m =3$  and  use $-NOT_2$, $-NOT_3$, and $(-NOT_2)(-NOT_3)$ for $m \geq 5$, acting on $|0100\rangle$.    Similarly, for $n=4$: use  $Z_1 Z_3 NOT_2$, $Z_2 Z_4NOT_3$, $Z_3Z_5NOT_4$, $(Z_1Z_3NOT_2)(Z_2Z_4NOT_3)$, $(Z_2Z_4NOT_3)(Z_3Z_5NOT_4)$, $(Z_2Z_4NOT_3)(Z_1Z_3NOT_2)$, and $(Z_3Z_5NOT_4)(Z_2Z_4NOT_3)$ for $m =3$ and use $-NOT_2$, $-NOT_3$, $-NOT_4$, $(-NOT_2)(-NOT_3)$, $(-NOT_3)(-NOT_4)$, $(-NOT_3)(-NOT_2)$, $(-NOT_4)(-NOT_3)$ for $m \geq 5$, acting on $|0110 \rangle$.
\end{proof} \bigskip
%%%%%%%%%%%%%%%%%%%%

\begin{lemma}
$\Gamma_{skl}$ is a finite abelian group, isomorphic to the product of $n(n-1)$ copies of $ \mathbb Z_m$.
\end{lemma}
\begin{proof}
In both the $m=3$ and $m\geq5$ odd cases, the given generators of $\Gamma_{skl}$ are distinct and no power of one is equal to the power of another.  In fact, the generators form a linearly independent set of $n(n-1)/2$ elements, as can be seen by their action on the qubit $| 0 0 \cdots 0 \rangle$.  Moreover, the generators commute with each other, and each generator has exactly order $m$.   So it must be the abelian product of $n(n-1)/2$ copies of $\mathbb Z_m$.
\end{proof} \bigskip

Putting all the lemmas together proves Theorem \ref{TheThm}, that for $m\geq  3$ odd, $Image(\rho_R) \cong \mathbb Z_m^{\frac{n(n-1)}{2}} \rtimes S_n.$


\begin{thebibliography}{AA}

\bibitem{CHW}S.\ X.\ Cui, S.M.\ Hong, and Z.\ Wang. \emph{Universal quantum computation with weakly integral anyons.} Quantum Information Processing (2014), 1-41.
\bibitem{CW}S.\ X.\ Cui, and Z.\ Wang. \emph{Universal quantum computation with metaplectic anyons.} Journal of Mathematical Physics 56.3 (2015), 032202.
\bibitem{HNW1} M.\ B.\ Hastings; C.\ Nayak; Z.\ Wang, \emph{On metaplectic modular categories and their applications.} Comm. Math. Phys. \textbf{330} (2014), no. 1, 45--68.
\bibitem{HNW2} M.\ B.\ Hastings; C.\ Nayak; Z.\ Wang, \emph{Metaplectic anyons, Majorana zero modes, and their computational power} Phys. Rev. B \textbf{87}, (2013), 165421.
\bibitem{KW}A.\ Kitaev, and Z.\ Wang, \emph{Solutions to generalized Yang-€"Baxter equations via ribbon fusion categories.} Geometry and Topology Monographs 18 (2012), 191-197.
%\bibitem{NR}D.\ Naidu, and E.\ C.\ Rowell. \emph{A finiteness property for braided fusion categories.} Algebras and representation theory 14.5 (2011): 837-855.
\bibitem{RZWG}E.\ C.\ Rowell, Y.\ Zhang, , Y.\ S.\ Wu, M.\ L.\ Ge, M. L. \emph{Extraspecial two-groups, generalized Yang-Baxter equations and braiding quantum gates.} Quantum Inf. Comput., 10(2010), 685-702.
%\bibitem{RW}E.\ C.\ Rowell, and H.\ Wenzl. \emph{$ SO (N) _2 $ Braid group representations are Gaussian.} arXiv preprint arXiv:1401.5329 (2014).


\end{thebibliography}
\end{document}